\documentclass[a4paper]{article}
\usepackage{appendix}
\usepackage{graphicx} 
\usepackage{amsmath, amsthm, amssymb}
\usepackage{mathtools}
\usepackage[backref=page]{hyperref}
\usepackage{faktor}

\renewcommand*{\backref}[1]{}
\renewcommand*{\backrefalt}[4]{%
    \ifcase #1 Not cited.%
    \or        Cited on page~#2.%
    \else      Cited on pages~#2.%
    \fi}

\let\phi\varphi 

\newtheorem{theorem}{Theorem}

\newtheorem{remark}[theorem]{Remark}
\newtheorem{corollary}[theorem]{Corollary}
\newtheorem{definition}[theorem]{Definition}
\newtheorem{lemma}[theorem]{Lemma}
\newtheorem{proposition}[theorem]{Proposition}

\newcommand{\NN}{\mathbb{N}}

\title{The Klein bottle has stably unbounded homeomorphism group}
\author{Lukas Böke\footnote{boeke@math.lmu.de}}
\date{}

\begin{document}

\maketitle

\begin{abstract}
Using a recent result of Bowden, Hensel and Webb, we prove the existence of a homeomorphism with positive stable commutator length in the group of homeomorphisms of the Klein bottle which are isotopic to the identity.
\end{abstract}

\section{Introduction}
In this article, we study the (un-)boundedness of norms on the group of homeomorphisms of the Klein bottle that are isotopic to the identity.
The study of questions of this type first appeared in the work of Burago, Ivanov and Polterovich \cite{BIP}, who considered the identity components of groups of compactly supported $C^\infty$-diffeomorphisms of various manifolds.
They proved that for all compact 3-manifolds and certain compact 4-manifolds these groups only admit conjugation-invariant norms that are bounded.
Their results have been extended to higher dimensional manifolds in \cite{tsuboi}.
These proofs use so-called displacement techniques, which have since been generalised, e.g.\ in \cite{cfflm} to prove the vanishing of bounded group cohomology.

In the case of compact surfaces, the situation is different:
by Bavard duality \cite{bavard}, the existence of non-trivial homogeneous quasi-morphisms implies that commutator length is unbounded.
Bowden, Hensel and Webb introduced the fine curve graph, which can be treated with the techniques of Bestvina and Fujiwara, who proved that the mapping class group of a hyperbolic surface allows infinitely many independent homogeneous quasi-morphisms \cite{Bestvina}.
For compact orientable surfaces, \cite{BHW1} and \cite{BHW2} show the existence of quasi-morphisms and therefore the existence of unbounded norms.
These techniques were adapted to non-orientable surfaces with non-orientable genus at least 3 in \cite{kim-kuno}.

We use a recent result from \cite{BHW2} to show that the group of homeomorphisms of the Klein bottle $K$ which are isotopic to the identity is stably unbounded with respect to commutator length in the sense of \cite{BIP}.

\paragraph{Acknowledgements.}
The author is deeply indebted to Sebastian Hensel for suggesting the topic and for helpful discussions.
The author also wants to thank the referee for their helpful comments.

\section{Background}

For any surface $S$, we define $\operatorname{Homeo}(S)$ to be the group of homeomorphisms of $S$ with compact support.
We denote the path-component of the identity by $\operatorname{Homeo}_0(S)$.
The \emph{mapping class group} of $S$ is then 
\[\operatorname{Mcg}(S) \coloneqq \faktor{\operatorname{Homeo}(S)}{\operatorname{Homeo}_0(S)}.\]
If $M\subset S$ is a finite subset, then we write $\operatorname{Homeo}(S,M)$ for the group of homeomorphisms that fix $M$ setwise, and call the elements of $M$ \emph{marked points}.
In the same way as before, we write $\operatorname{Homeo}_0(S,M)$ for the path-component of the identity, and $\operatorname{Mcg}(S,M)$ for the quotient.

We distinguish sets of points that are fixed pointwise from sets fixed setwise in the following way:
for a finite set $N\subset S$ that we want to be fixed pointwise, we consider the compactly supported homeomorphisms of the punctured manifold $S\setminus N$.
Once more, we have a path-component of the identity $\operatorname{Homeo}_0(S\setminus N,M)$ and a mapping class group $\operatorname{Mcg}(S\setminus N,M)$.
We can go back and forth between these different interpretations by restricting maps or extending them on a finite number of points.

The Birman exact sequence provides a description of some elements in the groups described above (see Theorems 4.6 and 9.1 in \cite{primer}):
the homomorphism Push maps homotopy classes of curves in the unordered configuration space $\operatorname{UConf}_m(S)$ of $m$ points in S to classes in $\operatorname{Mcg}(S,M)$.
These are called \emph{(multi-)point-pushes}.

Informally, these can be understood as grabbing a point of a surface, and moving it around along a path on the surface while stretching respectively compressing the surface around that point.
More formally, we can think of a (multi-) curve as a homotopy between inclusions of one or more points into the surface, and the $\operatorname{Push}$ map sending it to the homotopy class of its extensions to the whole surface.
Since Push is injective, we can make this visible by describing a left inverse on its image:
for a representative $f$ of $\phi \in \operatorname{im}(\operatorname{Push})$, we can find a path $f_t$ in $\operatorname{Homeo_0}(S)$ with $f_0 = \operatorname{id}$ and $f_1 = f$.
For $\chi(S) <0$, all such paths will be homotopic (see the series of papers by Hamstrom \cite{hamstrom1},\cite{hamstrom2},\cite{hamstrom3} as well as \cite{quintas}).
Evaluating this at one of the marked points $s_i$ gives us a path $\gamma_i \coloneqq f_t(s_i)$.
Combining these paths for $1\leq i \leq m$ gives us the desired loop in $\pi_1(\operatorname{UConf}_m(S),M)$.

We conclude this section by showing a useful lemma for dealing with multi-point-pushes:
\begin{lemma}\label{lem:properties-multi-pt-push}
Let $S$ be a (possibly non-orientable) surface with finitely many punctures, marked points $M \coloneqq \{s_1, \dotsc, s_m\}$ and $\chi(S) < 0$.
Fix a class $l\in \pi_1(\operatorname{UConf}_m(S), M)$ and its image $\operatorname{Push}(l)\in \operatorname{Mcg}(S,M)$.
Let $\pi: \widehat S \rightarrow S$ be a covering of degree $k \in \NN$, and denote lifts with respect to $\pi$ with a hat: $\widehat \quad$.
Let $\phi \in \operatorname{Mcg}(S,M)$.
Then:
\begin{enumerate}
    \item $\phi\operatorname{Push}(l)\phi^{-1} = \operatorname{Push}(\phi(l))$
    \item $ \widehat{\operatorname{Push}(l)} = \operatorname{Push}(\,\widehat l\,) $
\end{enumerate}
\end{lemma}

\begin{proof}
For both statements, we will argue that both mapping classes have the same image under the left inverse of Push described above.

First, let again $f$ be a representative of $\operatorname{Push}(l)$, and $f_t$ a path in  $\operatorname{Homeo_0}(S)$ with $f_0 = \operatorname{id}$ and $f_1 = f$.
For $\phi$ there is a permutation $\sigma: \{1,\dotsc,m\}\rightarrow \{1,\dotsc,m\}$ such that $\phi(s_i) = s_{\sigma(i)}$, as $\phi$ fixes the set $M$.
Let $g$ denote a representative of $\phi$.
Note that conjugating $f_t$ by $g$ gives us a path to a representative of $\phi\operatorname{Push}(l)\phi^{-1}$, which is therefore again in $\operatorname{Homeo_0}(S)$.
We can now consider the curves 
\[ (g\circ f_t \circ g^{-1})(s_i) =  (g\circ f_t )(s_{\sigma^{-1}(i)}) = g(\gamma_{\sigma^{-1}(i)}).\]
As we work with the fundamental group of an unordered configuration space, these curves represent $\phi(l)$, and therefore $g\circ f_t \circ g^{-1}$ is a representative of $\operatorname{Push}(\phi(l))$.

For the second statement, we first explain some of the lifts that appear in the statement:
the fibre over each marked point in $M$ consists of $k$ points, and a curve representing a component of $l$ starting at that point downstairs has $k$ lifts, one for each point in that fibre.
We denote the union of these fibres by $\widehat M$.
This means that we lift a curve in $\pi_1(\operatorname{UConf}_m(S), M)$ to a curve in $\pi_1(\operatorname{UConf}_{km}(\widehat S),\widehat M)$.
Evaluating a homotopy from the trivial element to a representative of a class in $\operatorname{Mcg}(\widehat S,\widehat M)$ at each marked point gives us a loop in $\operatorname{UConf}_{km}(\widehat S)$.
Comparison with the lift of $l$ gives the result.
\end{proof}

\section{Main Result}

We begin with the Klein bottle $K$, and introduce one puncture $q$ and one marked point $p\neq q$.
A class $v \in \pi_1(K\setminus \{q\}, p)$ induces a homeomorphism of the punctured Klein bottle via the homomorphism Push in the Birman exact sequence.

We can lift this setting to the torus $T$:
fix a covering map $\operatorname{pr}: T \rightarrow K$, and denote by $p_1,p_2$ the lifts of $p\in K$ with respect to $\operatorname{pr}$, and analogously $q_1,q_2$ for $q$, as well as $P\coloneqq \{p_1,p_2\}$ and $Q \coloneqq \{q_1,q_2\}$.
The class $v$ can be lifted as well, giving us two classes of curves:
the first starts at $p_1$ and is denoted by $v^{(1)}$, the second starts at $p_2$ and is denoted by $v^{(2)}$.
This pair defines a class of loops $\widehat v \coloneqq \{ v^{(1)},v^{(2)}\}$ in $\operatorname{UConf}_2(T\setminus Q)$.
Again, this loop defines a multi-point-push in the mapping class group of the torus with two punctures and two marked points via the homomorphism Push.

Our strategy will be to find a pseudo-Anosov point-push on the Klein bottle such that its lift to the torus satisfies the requirements of the following theorem.

\begin{theorem}[Theorem 5.2 in \cite{BHW2}]\label{thm:bhw-criterion}
Let $S$ be a closed, connected, oriented surface of genus at least 1.
Suppose $F$ is a Thurston representative of a multi-point-pushing pseudo-Anosov mapping class relative to its set of 1-prongs $P_1$ so that $[F]\in \operatorname{Mcg}(S, P_1)$ is not conjugate to its inverse.
Then $F$ has positive stable commutator length (scl) in $\operatorname{Homeo}(S)$.
\end{theorem}

The main result of this paper is the following:

\begin{theorem}\label{thm:not-conj-to-inv}
There exists a point-push $\operatorname{Push}(v) \in \operatorname{Mcg}(K\setminus \{q\}, \{p\})$ such that its lift $\widehat{\operatorname{Push}(v)}$ is not conjugate to its inverse in $\operatorname{Mcg}(T,Q\cup P)$.
\end{theorem}

The underlying reason for this is a statement on the fundamental group of the torus with two punctures.
In the following, we will call elements $a,b$ of a group $G$ \emph{automorphism-equivalent} if there exists $f \in \operatorname{Aut}(G)$ such that $f(a) = b$.
The fundamental group of $T\setminus Q$ is the free group with 3 generators, where an algorithm for deciding whether two elements are automorphism-equivalent is known:
this is Whitehead's algorithm, which we briefly introduce in Appendix \ref{app:whitehead}.
We can now prove:

\begin{proposition}\label{prop:not-inv-free-gp}
There exists an element $v \in \pi_1(K\setminus\{q\}, p)$ such that its lift to $\pi_1(T\setminus Q, p_1)$ is not automorphism-equivalent to its inverse.
\end{proposition}

\begin{proof}
We first exhibit a loop on the doubly punctured torus $T\setminus Q$ with the desired property, and then show that it is a lift of a curve on the Klein bottle.

Recall that $\pi_1(T\setminus Q,p_1) \cong F_3$ is the free group on three generators.
Using Whitehead's algorithm (see Appendix \ref{app:whitehead}) we identify $w \coloneqq a^2c^2ac^{-1}$ as not automorphism-equivalent to its inverse.
A script for GAP 4.12.2 and a log of a GAP session verifying that $w$ is not automorphism-equivalent to its inverse can be found in the accompanying Zenodo repository \cite{gap-repo}.
We keep $w$ as our candidate word for the rest of the article.

\begin{figure}
    \centering
    \includegraphics[width=\textwidth]{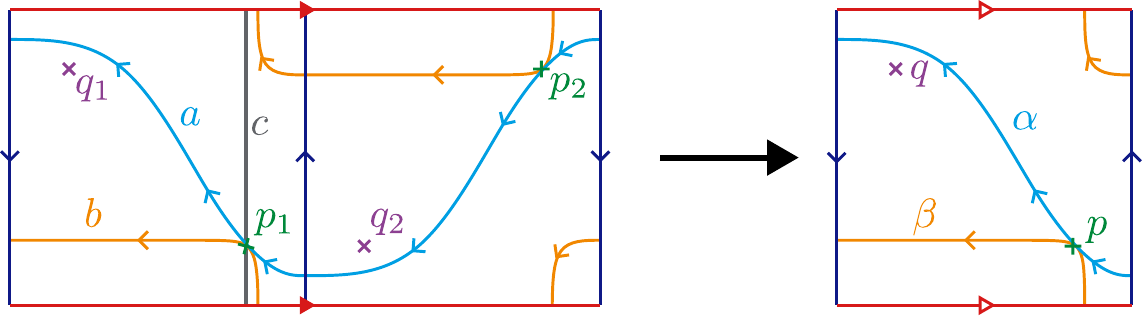}
    \caption{Setting with our choice of generators}
    \label{fig:generators}
\end{figure}

Figure \ref{fig:generators} fixes curves on the Klein bottle and torus that generate the appropriate groups.
By our choice of generators, we can see that our candidate word is the lift of the word $v \coloneqq \alpha^4 (\alpha \beta^{-1})^2 \alpha^2 \beta \alpha^{-1} $.
\end{proof}

\begin{remark}
Note that one can find many more elements of $\pi_1(T\setminus Q)$ with the desired properties, which we expect to lead to more, independent quasi-morphisms.
\end{remark}

We can now prove our main result:
\begin{proof}[Proof of Theorem \ref{thm:not-conj-to-inv}]
As in Proposition \ref{prop:not-inv-free-gp}, let $v \in \pi_1(K\setminus \{q\}, p)$ be our candidate word.
As above, we get a class  $\widehat v \coloneqq \{v^{(1)},v^{(2)}\}$ in $\pi_1( \operatorname{UConf}_2(T\setminus Q), P)$.
By Lemma \ref{lem:properties-multi-pt-push}, we have that $\widehat{\operatorname{Push}(v)}  = \operatorname{Push}(\widehat{v})$ in $\operatorname{Mcg}(T\setminus Q,P)$, which we now include into $\operatorname{Mcg}(T,Q\cup P)$.

Assume by contradiction that there exists $\phi \in \operatorname{Mcg}(T,Q\cup P)$ such that $\phi \operatorname{Push}(\widehat v) \phi^{-1}$ is the same class as $\operatorname{Push}(\widehat v^{-1})$.
Note that we can not immediately use Lemma \ref{lem:properties-multi-pt-push} here, as $\chi(T) =0$.

\emph{Claim:} we can assume that $\phi(Q) = Q$ and $\phi(p_1) = p_1$.
This claim implies that $\phi$ induces an automorphism of $\pi_1(T\setminus Q,p_1)$ that takes $v^{(1)} = w$ to its inverse, which is evident when evaluating the isotopies at $p_1$ as in the proof of Lemma \ref{lem:properties-multi-pt-push}.
But this contradicts Proposition \ref{prop:not-inv-free-gp}, so it remains to prove the claim to prove the theorem.

\emph{Proof of claim:}
We start by showing that $\phi$ can not send one point in $Q$ to $Q$ and the other one to $P$.
We achieve this by considering the images of $\phi \operatorname{Push}(\widehat v) \phi^{-1}$ and $\operatorname{Push}(\widehat v^{-1})$ under the natural homomorphism that forgets two of the punctures.
Note that these mapping classes both fix the points in $P\cup Q$ pointwise.
We denote this forgetful homomorphism by 
\[\operatorname{Keep}_{x,y}:\operatorname{Mcg}(T\setminus (P\cup Q))\rightarrow \operatorname{Mcg}(T\setminus \{x,y\})\]
to indicate that we forget all points but $x, y$ in $P\cup Q$.
Note that if we keep the two points in $\phi(Q)$, then the image of $\phi \operatorname{Push}(\widehat v) \phi^{-1}$ must be trivial, as we forget those points which are pushed.

Now assume that $\phi(Q) = \{q_i,p_j\}$ for some $i,j \in \{1,2\}$.
We can then check that $\operatorname{Keep}_{q_i,p_j}\left( \operatorname{Push}(\widehat v^{-1}) \right)$ is not trivial.
Hence, we have either $\phi(Q) = Q$ or $\phi(Q) = P$.
In the case that $\phi(Q) = P$, we use a similar argument, but we consider the images under $\operatorname{Keep}_{p_1,p_2}$.
We can check for non-triviality by applying $\operatorname{Push}(\widehat v^{-1})$ to a curve in the free group $\pi_1(T\setminus(P\cup Q),\ast)$ and finding that the original and the pushed curve, after forgetting two of the punctures, are different.
An implementation of this in GAP can be found in \cite{gap-repo}, including an illustration showing our choice of generators for $\pi_1(T\setminus(P\cup Q),\ast)$.

Now assume that $\phi(Q) = Q$ and $\phi(p_1) = p_2$.
Then, we can replace $\phi$ by $\phi \circ \tau$, where $\tau$ is the generator of the deck group of $T\rightarrow K$.
Note that $\tau \operatorname{Push}(\widehat v) \tau^{-1} = \operatorname{Push}(\widehat v)$ since $\operatorname{Push}(\widehat v)$ is a lift.
This concludes the proof of the theorem.
\end{proof}

We use the following criterion to determine whether a point-push is pseudo-Anosov:
\begin{theorem}[Kra's Construction, Theorem 14.6 in \cite{primer}]\label{thm:kra}
Let $S$ be a compact orientable surface with $\chi(S) <0$, and let $c \in \pi_1(S,p)$.
The mapping class $\operatorname{Push}(c)\in \operatorname{Mcg}(S,\{p\})$ is pseudo-Anosov if and only if $c$ fills $S$.
\end{theorem}

\begin{remark}\label{rem:non-or-kra}
Note that this theorem is stated for orientable surfaces.
But the proof given in \cite{primer} applies to the punctured Klein bottle as well, and hence $\operatorname{Push}(v)$ is a pseudo-Anosov.
This relies on a non-orientable version of the Nielsen-Thurston classification as found in \cite{wu} and \cite{non-orientable-classification}.
\end{remark}

We now prove the result in the title:

\begin{corollary}
The group $\operatorname{Homeo}_0(K)$ with the commutator length is stably unbounded.
In particular, there exists a non-trivial homogeneous quasi-morphism $\operatorname{Homeo}_0(K) \rightarrow \mathbb{R}$.
\end{corollary}

\begin{proof}
Take $\operatorname{Push}(v)$ and $\operatorname{Push}(\widehat v)$ as before.
First, $\operatorname{Push}(v)$ is pseudo-Anosov by Remark \ref{rem:non-or-kra}, and its lift $\operatorname{Push}(\widehat v)$ is pseudo-Anosov as well.
This follows as we can consider a Thurston representative of $\operatorname{Push}(v)$, lift its invariant foliations, and find that in the lifted setting, all conditions for being pseudo-Anosov are still satisfied.
Next, we need to check that a Thurston representative of $\operatorname{Push}(v)$ has 1-prongs at $p$ and $q$.
Note that the Euler-Poincar\'e formula (see chapter 11 of \cite{primer} or Expos\'e 5 of \cite{flp}) requires us to have at least one 1-prong.
We consider the image of $\operatorname{Push}(v)$ under $\operatorname{Keep}_x$ for $x\in \{p,q\}$.
One can check that any point-push in $\operatorname{Mcg}(K\setminus\{x\})$ fixes an essential simple closed curve, which contradicts having only one 1-prong at $x$, as this would imply that this point-push is pseudo-Anosov.
As this argument is independent of the choice of $x$, we get that a representative of $\operatorname{Push}$ has 1-prongs at $p$ and $q$.
By Theorem \ref{thm:not-conj-to-inv}, $\operatorname{Push}(\widehat v)$ is not conjugate to its inverse in $\operatorname{Mcg}(T, Q\cup P)$.
As the conditions of Theorem \ref{thm:bhw-criterion} are satisfied, we get that a Thurston representative of $\operatorname{Push}(\widehat v)$ has positive scl in $\operatorname{Homeo}(T)$.

Let $F$ be a Thurston representative of $\operatorname{Push}(v)$.
Then, $F$ lifts to a Thurston representative $\widehat{F}$ of $\operatorname{Push}(\widehat v)$ (cf.\  Expos\'e 12 in \cite{flp}) which has positive scl.
Assuming that $\operatorname{scl}(F)=0$ leads to a contradiction:
for each power of $F$, assume we have a factorisation $F^n = \prod_{i=1}^{a_n} [f_i^{(n)},g_i^{(n)}] $ for some sequence $a_n$ such that $\lim_{n\rightarrow \infty} a_n/n = 0$.
We note that each of $F^n$, $f_i^{(n)}$ and $g_i^{(n)}$ has two possible lifts to the torus, which differ by the non-trivial deck transformation, and all these lifts commute with the deck transformation.
Also note that $\operatorname{scl}(\tau \circ \widehat{F}) = \operatorname{scl}(\widehat F)$.
Any such factorisation of $F^n$ can therefore be lifted to a factorisation of $\widehat F^n$ or $\tau \circ \widehat F^n$ into commutators, contradicting $\operatorname{scl}(\widehat F) > 0$.
\end{proof}

\begin{remark}
Note that the restriction of a homogeneous quasi-morphism on $\operatorname{Homeo}(K)$ to diffeomorphisms is still a homogeneous quasi-morphism.
As diffeomorphisms are dense (see e.g. Chapter 2 of \cite{hirsch}), we can use the automatic continuity of homogeneous quasi-morphisms on homeomorphism groups (see the appendix of \cite{BHW1}), which shows that its restriction to $\operatorname{Diff}_0(K)$ is still an unbounded quasi-morphism.
\end{remark}

\bibliographystyle{alpha}
\bibliography{refs}

\newpage

\appendix
\section{Whitehead's Algorithm}\label{app:whitehead}
In this appendix, we briefly explain our use of Whitehead's algorithm.
The origins of this algorithm are Whitehead's papers \cite{whitehead-1}, \cite{whitehead-2}.
The proof of Whitehead's theorem has been simplified over time, one such simplification was given in \cite{higgins-lyndon} by Higgins and Lyndon.
In our implementation, we also incorporate some observations from \cite{lau} which simplify computations; \cite{lau} is part of the proceedings of a REU program at Oregon State University that took place in 1997.

Let $S$ be a finite set, and $F \coloneqq F(S)$ the free group generated by $S$, and set $L \coloneqq S \cup S^{-1}$.
We start by defining the two types of Whitehead automorphisms.

\begin{definition}
A permutation of $L$ which extends to an element of $\operatorname{Aut}(F)$ is called a \emph{Type I Whitehead automorphism}.

For a fixed $a\in L$, the \emph{Type II Whitehead automorphisms} send a generator $x$ to one of $x$, $xa$, $a^{-1}x$ or $a^{-1}xa$.
\end{definition}

The main ingredient of Whitehead's algorithm is the following theorem:
\begin{theorem}[Whitehead, 1936]\label{thm:whitehead}
    If $w,w' \in F$ are automorphism-equivalent such that $w'$ has minimal length for their automorphism-equivalence class, then there is some sequence $\alpha_1,\dotsc, \alpha_n$ of Whitehead automorphisms such that \\$(\alpha_n \circ \dotsm \circ \alpha_1)(w) = w'$, and for $1 \leq k \leq n$, we have that 
    \[\operatorname{length}((\alpha_k \circ \dotsm \circ \alpha_1)(w)) \leq \operatorname{length}((\alpha_{k-1} \circ \dotsm \circ \alpha_1)(w))\]
    with strict inequality whenever $\operatorname{length}((\alpha_{k-1} \circ \dotsm \circ \alpha_1)(w))$ is not minimal.
\end{theorem}

To make computations easier, we make two observations:
\begin{enumerate}
    \item Type I Whitehead automorphisms don't affect word length.
    \item We can assume that in a sequence $\alpha_1,\dotsc,\alpha_n$ of Whitehead automorphisms as in the theorem, the only Type I automorphism is $\alpha_n$.
\end{enumerate}

Next, we give a short description of our implementation found in the accompanying Zenodo repository \cite{gap-repo}:
let $w,w' \in F$ as in Theorem \ref{thm:whitehead}.
We restrict to the case that we have two words which have minimal length within their automorphism-equivalence class.
Next, we repeatedly apply all possible Type II automorphisms to $w$ to compile a list of distinct automorphism-equivalent words of minimal length.
We complete the list by applying all possible Type I automorphisms to the words in our list to find a complete list of words which are automorphism-equivalent to $w$ and of minimal length. 
We get that $w$ and $w'$ are automorphism-equivalent if and only if $w'$ appears in this list.

\end{document}